\documentclass[12pt,reqno]{amsart}

\newcommand\version{June 14, 2012}

\usepackage{amsmath,amsfonts,amsthm,amssymb,amsxtra}

\newtheorem{theorem}{Theorem}[section]

\newtheorem{lemma}[theorem]{Lemma}
\newtheorem{corollary}[theorem]{Corollary}

\theoremstyle{definition}

\newtheorem{example}[theorem]{Example}

\theoremstyle{remark}

\numberwithin{equation}{section}

\newcommand{\B}{\mathfrak{B}}

\newcommand{\C}{\mathbb{C}}

\newcommand{\const}{\mathrm{const}\ }

\renewcommand{\epsilon}{\varepsilon}

\newcommand{\loc}{{\rm loc}}

\newcommand{\R}{\mathbb{R}}

\newcommand{\Sph}{\mathbb{S}}
\renewcommand{\S}{\mathfrak{S}}
\renewcommand{\sc}{\mathrm{sc}}

\DeclareMathOperator{\tr}{Tr}

\begin{document}

\title[Cwikel's theorem --- \version]{Cwikel's theorem and the CLR inequality}

\author{Rupert L. Frank}
\address{Rupert L. Frank, Department of Mathematics, Princeton University, Princeton, NJ 08544, USA}
\email{rlfrank@math.princeton.edu}

\thanks{\copyright\, 2012 by the author. This paper may be reproduced, in its entirety, for non-commercial purposes.\\
U.S. National Science Foundation grant PHY-1068285 is acknowledged. The author is grateful to A.~Laptev, M. Lewin, E. Lieb, R. Seiringer and T. Weidl for helpful discussions.}

\begin{abstract}
We give a short proof of the Cwikel--Lieb--Rozenblum (CLR) bound on the number of negative eigenvalues of Schr\"odinger operators. The argument, which is based on work of Rumin, leads to remarkably good constants and applies to the case of operator-valued potentials as well. Moreover, we obtain the general form of Cwikel's estimate about the singular values of operators of the form $f(X) g(-i\nabla)$.
\end{abstract}

\maketitle

%%%%%%%%%%%%%%%%%%%%%%%%%%%%%%%%%%%%%%%%%%%%%%%

\section{Introduction}

Among the most beautiful theorems in spectral theory is Cwikel's result about trace ideal properties of operators on $L_2(\R^d)$ of the form $f(X)g(-i\nabla)$. Here $f(X)$ denotes multiplication by the function $f$ in position space and $g(-i\nabla)$ denotes multiplication by $g$ in momentum space. Cwikel's theorem says that $f\in L_q(\R^d)$ and $g\in L_{q,w}(\R^d)$ implies $f(X) g(-i\nabla)\in \mathfrak{S}_{q,w}(L_2(\R^d))$ for $q>2$. (We recall the definition of weak $L_q$ and weak $\mathfrak S_q$ spaces below.) This was conjectured by Simon in \cite{Si} and proved by Cwikel in \cite{Cw}; see also the review \cite{BiKaSo} for some extensions of this result.

An immediate consequence of Cwikel's theorem is the famous Cwikel--Lieb--Rozen\-blum bound on the number $N(0,-\Delta+V)$ of negative eigenvalues (counting multiplicities) of Schr\"odinger operators $-\Delta+V$ in $L_2(\R^d)$, $d\geq 3$, namely,
\begin{equation}
\label{eq:clr}
N(0,-\Delta+V) \leq \const \int_{\R^d} V(x)_-^{d/2} \,dx \,.
\end{equation}
Here $V(x)_-=\max\{-V(x),0\}$ denotes the negative part. The meaning of this bound is that the semi-classical approximation,
$$
\iint_{\R^d\times\R^d} \chi_{\{ p^2 + V(x)<0 \}} \frac{dx\,dp}{(2\pi)^d} = (2\pi)^{-d} |\{ p\in\R^d:\ |p|<1\}| \int_{\R^d} V(x)_-^{d/2} \,dx \,,
$$
is, indeed, a uniform upper bound on $N(0,-\Delta+V)$ up to a universal constant depending only on the dimension. Different proofs of \eqref{eq:clr} were given in \cite{Ro,Li,Fe,LiYa,Co}; see also the reviews \cite{LaWe2,H2}.

One of our goals here is to provide a new and simple proof of Cwikel's theorem and the CLR inequality. Our starting point is the remarkable paper \cite{Ru1} by Rumin which contains, among others, the inequality
\begin{equation}
\label{eq:ruminrd}
\tr\gamma^{1/2}(-\Delta)\gamma^{1/2} \geq \const \int_{\R^d} \gamma(x,x)^{d/(d-2)} \,dx
\end{equation}
for operators $0\leq\gamma\leq (-\Delta)^{-1}$ on $L_2(\R^d)$, $d\geq 3$. As we shall see, this is a very powerful inequality (for instance, for $\gamma$ of rank one, it reduces to Sobolev's inequality). Surprisingly, its proof is elementary and uses not much more than the triangle inequality for the Hilbert--Schmidt norm. It also yields a rather good value for the constant. In this paper we shall derive the CLR inequality \eqref{eq:clr} from \eqref{eq:ruminrd} and we shall extend \eqref{eq:ruminrd} to $L_2(\R^d)\otimes\mathcal G$ with constants independent of the dimension of the auxiliary Hilbert space $\mathcal G$. Both results are new and go beyond \cite{Ru1,Ru}. Our results in the operator-valued case improve upon previous results of \cite{H1} (who follows \cite{Cw} and has larger constants) and \cite{FrLiSe1} (who can only deal with $(-\Delta)^s$ for $0<s\leq 1$). Moreover, we show that a modification of Rumin's proof of \eqref{eq:ruminrd} yields an easy proof of Cwikel's theorem mentioned at the beginning. This is the topic of Section \ref{sec:cwik}.

Besides its simplicity and its good constants, another advantage of Rumin's inequality \eqref{eq:ruminrd} is that it is not limited to the Laplacian (or its powers) on $\R^d$, but has extensions to a large class of abstract operators. Roughly speaking, the only assumption is the existence of a density of states, and the energy dependence of this density of states determines the way in which $\gamma(x,x)$ enters the right side of \eqref{eq:ruminrd}. This generality of \cite{Ru1,Ru} was of crucial importance for the results in \cite{FrOl,FrLeLiSe}. In this paper we do not aim at highest possible generality, but we do include a new theorem about operators $T$ on arbitrary measure spaces $X$. We prove that a diagonal heat kernel bound $\exp(-tT)(x,x) \leq C t^{-\nu/2}$ with $\nu>2$ implies a CLR inequality $N(0,T+V) \leq C_\nu' C \int_X V_-^{\nu/2} \,dx$; see Theorem \ref{cwikelgen}. This improves earlier results in \cite{LeSo,FrLiSe2} who needed the additional assumption that $\exp(-tT)$ is positivity preserving.

In addition to deriving the CLR inequality \eqref{eq:clr} from \eqref{eq:ruminrd} we are able to answer the following conceptual question about \eqref{eq:ruminrd}. Namely, besides the new inequality \eqref{eq:ruminrd} Rumin's papers \cite{Ru1,Ru} contain a new proof of the inequality
\begin{equation}
\label{eq:ltdens}
\tr\gamma^{1/2}(-\Delta)\gamma^{1/2} \geq \const \int_{\R^d} \gamma(x,x)^{(d+2)/d} \,dx
\end{equation}
for operators $\gamma$ on $L_2(\R^d)$ satisfying $0\leq\gamma\leq 1$. Inequality \eqref{eq:ltdens} is due to Lieb and Thirring \cite{LiTh} and plays an important role in their proof of stability of matter. It is well known that \eqref{eq:ltdens} is equivalent to an inequality about eigenvalue sums of Schr\"odinger operators, namely,
\begin{equation}
\label{eq:ltev}
\tr\left(-\Delta+V\right)_- \leq \const \int_{\R^d} V(x)_-^{(d+2)/2} \,dx \,.
\end{equation}
By `equivalence' we mean that there is a duality principle between \eqref{eq:ltdens} and \eqref{eq:ltev} and that the optimal constant for one inequality determines that for the other inequality. Given the striking similarity between \eqref{eq:ruminrd} and \eqref{eq:ltdens} it is natural to ask whether there is an inequality for Schr\"odinger operators which is equivalent to \eqref{eq:ruminrd}. We are able to answer this question completely (Lemma \ref{dual}) and see that \eqref{eq:ruminrd} is `essentially' equivalent to the CLR inequality. More precisely, we prove that \eqref{eq:ruminrd} is equivalent to a bound on the Birman--Schwinger operator $(-\Delta)^{-1/2} V_- (-\Delta)^{-1/2}$ in the weak trace ideal $\mathfrak S_{d/2,w}(L_2(\R^d))$, however, not with its standard (quasi-)norm but with an equivalent expression (Lemma \ref{equiv}).

For the impatient reader who wants to see immediately and without going through various dualities how \eqref{eq:ruminrd} implies the CLR bound \eqref{eq:clr} we finish this introduction with a short derivation of \eqref{eq:clr}. For fixed $\epsilon>0$ we know that the spectrum of $-\Delta+V$ in the interval $(-\infty,-\epsilon)$ is finite if $V_-\in L_{d/2}(\R^d)$. Let $\psi_1,\ldots,\psi_N$ be linearly independent functions which span the corresponding spectral subspace. Our goal will be to prove an upper bound on $N$ in terms of $V$, independently of $\epsilon$. We may assume that the functions are normalized so that $\langle\sqrt{-\Delta}\psi_j,\sqrt{-\Delta}\psi_k\rangle =\delta_{jk}$. Note that with this normalization, the $\psi_j$'s are linear combinations of eigenfunctions but, in general, not eigenfunctions. Since they span the spectral subspace of $-\Delta+V$ corresponding to $(-\infty,-\epsilon)$ we know, however, that $\gamma=\sum_j |\psi_j\rangle\langle\psi_j|$ satisfies
\begin{equation}
\label{eq:clrproof1}
0 \geq \tr\gamma^{1/2}(-\Delta+V)\gamma^{1/2} \,.
\end{equation}
Because of the normalization of the $\psi_j$'s we also know that
$$
\sqrt{-\Delta}\gamma\sqrt{-\Delta} \leq 1
\qquad\text{and}\qquad
\tr\gamma^{1/2}(-\Delta)\gamma^{1/2} = N \,.
$$
Thus, we infer from \eqref{eq:ruminrd} that
$$
N \geq K_d \int_{\R^d} \gamma(x,x)^{d/(d-2)} \,dx
$$
for some constant $K_d$ and, therefore, that
\begin{align*}
\tr V\gamma & \geq - \int_{\R^d} V(x)_- \gamma(x,x)\,dx 
\geq - \left( \int_{\R^d} V_-^{d/2} \,dx \right)^{2/d} \left( \int_{\R^d} \gamma(x,x)^{d/(d-2)} \,dx \right)^{(d-2)/d} \\
& \geq - N^{(d-2)/d} K_d^{-(d-2)/d} \left( \int_{\R^d} V_-^{d/2} \,dx \right)^{2/d} \,.
\end{align*}
We insert this bound into \eqref{eq:clrproof1} and get
$$
0 \geq \tr\gamma^{1/2}(-\Delta+V)\gamma^{1/2} = N + \tr V\gamma
\geq N - N^{(d-2)/d} K_d^{-(d-2)/d} \left( \int_{\R^d} V_-^{d/2} \,dx \right)^{2/d} \,.
$$
Thus,
$$
N \leq K_d^{-(d-2)/2} \int_{\R^d} V_-^{d/2} \,dx \,,
$$
independently of $\epsilon$, which proves \eqref{eq:clr}.

%%%%%%%%%%%%%%%%%%%%%%%%%%%%%%%%%%%%%%%%%%%%%%%

\section{Cwikel's theorem}\label{sec:cwik}

To state our main result we recall that $L_{p,w}(\R^d)$ denotes the space of functions $a$ on $\R^d$ for which the (quasi-)norm
$$
\| a\|_{p,w}^{p} = \sup_{\tau>0} \tau^{p}\, |\{ |a|>\tau \} |
$$
is finite. We shall prove

\begin{theorem}
\label{main}
Let $0\leq a\in L_{p,w}(\R^d)$ and $0\leq b\in L_p(\R^d)$ for some $p>1$. Then for all $\mu>0$
$$
\tr\left( a(-i\nabla)^{1/2} b(X) a(-i\nabla)^{1/2} - \mu\right)_+  \leq \mu^{-p+1} \ \frac{(p+1)^{p-1}}{(p-1)^p} \ (2\pi)^{-d} \ \| a\|_{p,w}^{p} \| b\|_p^p \,.
$$
\end{theorem}

Before proving this result we shall show that Cwikel's theorem is an easy consequence of it. We recall that $\S_{q,w}(\mathcal H)$ is the space of compact operators $K$ on a separable Hilbert space $\mathcal H$ satisfying
$$
\| K \|_{q,w}^q = \sup_{\kappa>0} \kappa^{q}\, n(\kappa, (K^* K)^{1/2} ) <\infty \,.
$$
Here $n(\kappa,(K^* K)^{1/2})$ denotes the number of eigenvalues of $(K^* K)^{1/2}$ larger than $\kappa$, counting multiplicities.

\begin{corollary}
[Cwikel's theorem]
\label{cwikel}
If $f\in L_{q}(\R^d)$ and $g\in L_{q,w}(\R^d)$ for some $q>2$, then $f(X)\,g(-i\nabla)\in\S_{q,w}( L_2(\R^d) )$ with
$$
\| f(X)\, g(-i\nabla) \|_{q,w}^q \leq \left( \frac{q}{q-2} \right)^{q/2} \left( \frac{q+2}{q-2} \right)^{(q-2)/2} (2\pi)^{-d} \ \| f \|_{q}^{q} \ \| g\|_{q,w}^q \,.
$$
\end{corollary}

In order to deduce Corollary \ref{cwikel} from Theorem \ref{main} we use the following lemma, which shows that the quantity bounded in Theorem \ref{main} is indeed equivalent to the norm in a weak Schatten class.

\begin{lemma}
[Equivalent quasi-norms]
\label{equiv}
Let $K$ be a compact operator on a separable Hilbert space $\mathcal H$ and let $q>2$. Then $K\in \S_{q,w}(\mathcal H)$ iff
$$
|K |_q' := \left( \sup_{\mu>0} \mu^{q/2-1} \tr(K^*K-\mu)_+ \right)^{1/q} <\infty\,.
$$
Moreover,
$$
|K|_q' \leq \left(\frac{2}{q-2}\right)^{1/q} \|K\|_{q,w}
\leq \left(\frac{q}{q-2}\right)^{1/2} |K|_q' \,.
$$
\end{lemma}

\begin{proof}
Since $(E-\mu)_+ = \int_\mu^\infty \chi_{(\sigma,\infty)}(E) \,d\sigma$ we have
$$
\tr(K^*K-\mu)_+ = \int_\mu^\infty n(\sqrt\sigma,(K^*K)^{1/2}) \,d\sigma \,.
$$
If $\|K\|_{q,w}$ is finite, this is bounded by
$$
\int_\mu^\infty n(\sqrt\sigma,(K^*K)^{1/2}) \,d\sigma \leq \|K\|_{q,w}^q \int_\mu^\infty \sigma^{-q/2} \,d\sigma = (q/2-1)^{-1} \|K\|_{q,w}^q \mu^{-q/2+1} \,.
$$
Thus, $|K|_q' \leq (q/2-1)^{-1/q} \|K\|_{q,w}$. Conversely, since $\chi_{(\kappa^2,\infty)}(E) \leq (\kappa^2-\mu)^{-1} (E-\mu)_+$ for any $\mu<\kappa^2$ we have
$$
n(\kappa,(K^*K)^{1/2}) \leq (\kappa^2-\mu)^{-1} \tr(K^*K-\mu)_+ \,.
$$
If $|K |_q'$ is finite, this is bounded by
$$
(\kappa^2-\mu)^{-1} \tr(K^*K-\mu)_+ \leq (\kappa^2-\mu)^{-1} \mu^{-q/2+1} \left( |K|_q' \right)^q \,.
$$
We optimize the right side by choosing $\mu=(1-2/q)\kappa^2$ and obtain
$$
n(\kappa,(K^*K)^{1/2}) \leq \frac{q}{2} \left(1-\frac2q\right)^{-q/2+1} \kappa^{-q} \left( |K|_q' \right)^q \,,
$$
that is, $\|K\|_{q,w} \leq \left(q/2\right)^{1/q} \left(1-2/q\right)^{-1/2+1/q} |K|_q'$, as claimed.
\end{proof}

\begin{proof}[Proof of Corollary \ref{cwikel}]
After applying unitaries, we may assume that $f$ and $g$ are non-negative. We put $K=f(X)g(-i\nabla)$. Applying Theorem \ref{main} with $a=g^2$, $b=f^2$ and $p=q/2$ we infer that
$$
\left( \| K \|_q'\right)^q = \sup_{\mu>0} \mu^{q/2-1} \tr(K^*K-\mu)_+ \leq \frac{\left(p+1\right)^{p-1}}{\left(p-1\right)^p} (2\pi)^{-d} \ \| a\|_{p,w}^{p} \| b\|_p^p \,.
$$
Lemma \ref{equiv} allows to turn this into a bound for $\|K\|_{q,w}$, which is the statement of Corollary \ref{cwikel}.
\end{proof}

We now turn to the proof of Theorem \ref{main}. The variational principle for sums of eigenvalues allows us to reformulate it in a dual form, in which we shall actually prove it. The precise statement is the following. (As usual, we write $p'=p/(p-1)$.)

\begin{lemma}
[Duality]
\label{dual}
Let $A$ be a non-negative operator in $L_2(X)$ (where $X$ is a sigma-finite measure space) with $\ker A=\{0\}$ and let $p>1$. Then the following inequalities are equivalent,
\begin{enumerate}
\item[(i)] $\tr\left( A^{1/2} b A^{1/2} - \mu\right)_+ \leq D \mu^{-p+1} \int_X b^p \,dx$ for every $0\leq b\in L_p(X)$ and $\mu>0$,
\item[(ii)] $\tr\gamma^{1/2}A^{-1}\gamma^{1/2} \geq K \int_X \gamma(x,x)^{p'} \,dx$ for every operator $0\leq \gamma\leq A$,
\end{enumerate}
in the sense that the optimal constants $D$ and $K$ are related by
$$
\left( p\ D\right)^{p'} \left(p'\ K \right)^{p} = 1 \,.
$$
\end{lemma}

\begin{proof}
This is a consequence of the variational characterization for the expression on the left side of (i), namely,
$$
\tr\left( A^{1/2} b A^{1/2} - \mu\right)_+ = \sup_{0\leq\delta\leq 1} \tr \delta^{1/2} \left( A^{1/2} b A^{1/2} - \mu\right)\delta^{1/2} \,.
$$
To prove that (ii) implies (i) we change variables from $\delta$ to $\gamma=A^{1/2}\delta A^{1/2}$. Then the conditions $0\leq\delta\leq 1$ imply that $0\leq\gamma\leq A$, and therefore by (ii),
$$
\tr \delta^{1/2} \left( A^{1/2} b A^{1/2} - \mu\right)\delta^{1/2} 
=\tr \gamma^{1/2} \left(b - \mu A^{-1} \right)\gamma^{1/2}
\leq \int_X \left( b \rho - \mu K \rho^{p'} \right) \,dx \,,
$$
where $\rho(x)=\gamma(x,x)$. Maximizing the right side over all functions $\rho\geq 0$ (i.e., ignoring the fact that $\rho$ was related to $\gamma$) we find that
$$
\int_X \left( b \rho - \mu K \rho^{p'} \right) \,dx \leq (K\mu)^{-p+1} \frac{(p-1)^{p-1}}{p^p} \int_X b^p \,dx \,,
$$
i.e., (i) holds and the optimal constant satisfies $D \leq K^{-p+1} \frac{(p-1)^{p-1}}{p^p}$. The proof of the converse implication is similar and is omitted.
\end{proof}

We now prove the dual form of Theorem \ref{main}. As we mentioned in the introduction, the proof follows closely some ideas of Rumin \cite{Ru1,Ru}.

\begin{lemma}
\label{rumin}
Let $a\in L_{p,w}(\R^d)$ with $p>1$ and assume that $a>0$ a.e. Then for any operator $\gamma$ on $L_2(\R^d)$ satisfying $0 \leq \gamma \leq a(-i\nabla)$, we have
$$
\tr\gamma^{1/2}a(-i\nabla)^{-1}\gamma^{1/2} \geq \frac{p-1}{p+1}\ R_{d,p}^\sc\  \| a\|_{p,w}^{-p'} \int_{\R^d} \gamma(x,x)^{p'} \,dx
$$
where $R_{d,p}^\sc = (2\pi)^{d/(p-1)} \left(\frac{p-1} p \right)^{p/(p-1)}$.
\end{lemma}

The superscript `sc' in $R_{d,p}^\sc$ stands for `semi-classical'. This will be further explored in Section \ref{sec:concl}.

It is part of the assertion that the assumption $\tr\gamma^{1/2}a(-i\nabla)^{-1}\gamma^{1/2}<\infty$ implies that the diagonal $\gamma(x,x)$ makes sense for a.e. $x\in\R^d$ and belongs to $L_{p'}(\R^d)$. Note that this diagonal value is well-defined if $\gamma$ is a finite rank operator. Given the bound from the lemma in this case, which is independent of the (finite) rank, the extension to general $\gamma$ can be carried out, for instance, by monotone convergence. We omit the details since the finite rank version is all we need for the proof of Theorem \ref{main}.

\begin{proof}
Since $E^{-1}=\int_0^\infty \chi_{(0,\tau]} (E) \tau^{-2} \,d\tau$, the spectral theorem together with Fubini's theorem implies that
\begin{equation}
\label{eq:repr}
\tr\gamma^{1/2}a(-i\nabla)^{-1}\gamma^{1/2} = \int_0^\infty \tr\gamma_\tau \, \frac{d\tau}{\tau^2} = \int_{\R^d} \int_0^\infty \rho_\tau(x)\, \frac{d\tau}{\tau^{2}} \,dx \,,
\end{equation}
where $\gamma_\tau = \chi_{(0,\tau]}(a(-i\nabla))\, \gamma\, \chi_{(0,\tau]}(a(-i\nabla))$ and where $\rho_\tau(x)=\gamma_\tau(x,x)$ is its density.

Our next goal is to find a pointwise lower bound on $\rho_\tau$ in terms of $\rho$. To do this, let $\Omega\subset\R^d$ be any set of finite measure and note that
\begin{align*}
\left( \int_\Omega \rho(x) \,dx \right)^{1/2} \!\! = \| \gamma^{1/2}\chi_\Omega \|_2 
\leq \| \gamma^{1/2} \chi_{(0,\tau]}(a(-i\nabla)) \ \chi_\Omega \|_2 + \| \gamma^{1/2} \chi_{(\tau,\infty)}(a(-i\nabla)) \ \chi_\Omega \|_2,
\end{align*}
where $\|\cdot\|_2$ is the Hilbert--Schmidt norm. The first term on the right side is
$$
\| \gamma^{1/2} \chi_{(0,\tau]}(a(-i\nabla)) \ \chi_\Omega \|_2 = \| \gamma_\tau^{1/2} \ \chi_\Omega \|_2
= \left( \int_\Omega \rho_\tau(x) \,dx \right)^{1/2} \,,
$$
and the second term, since $\gamma\leq a(-i\nabla)$, is bounded from above by
\begin{align*}
\| \gamma^{1/2} \chi_{(\tau,\infty)}(a(-i\nabla))\ \chi_\Omega \|_2 & \leq \| a(-i\nabla)^{1/2} \chi_{(\tau,\infty)}(a(-i\nabla))\ \chi_\Omega \|_2 \\
& = |\Omega|^{1/2} \left( \int_{\R^d} a(p) \chi_{\{a>\tau\}}(p) \frac{dp}{(2\pi)^d} \right)^{1/2} \,.
\end{align*}
Since $a(p)=\int_0^\infty \chi_{\{ a >\sigma \}}(p) \,d\sigma$ and since $| \{a> t\} | \leq \|a\|_{p,w}^p t^{-p}$, we find
\begin{align*}
\int_{\R^d} a(p) \chi_{\{a>\tau\}}(p) \,dp & = \int_0^\infty \int_{\R^d}  \chi_{\{a>\sigma\}}(p) \chi_{\{a>\tau\}}(p) \,dp\,d\sigma
= \int_0^\infty | \{ a> \max\{\sigma,\tau\} \} | \,d\sigma \\
& \leq \|a\|_{p,w}^p \int_0^\infty \min\{\sigma^{-p},\tau^{-p}\} \,d\sigma = \frac p{p-1} \| a\|_{p,w}^p \, \tau^{-p+1} \,.
\end{align*}
Thus, we have shown that
$$
\left( \int_\Omega \rho(x) \,dx \right)^{1/2} 
\leq \left( \int_\Omega \rho_\tau(x) \,dx \right)^{1/2} + |\Omega|^{1/2} (2\pi)^{-d/2} \left(\frac p{p-1} \right)^{1/2} \| a\|_{p,w}^{p/2} \, \tau^{-(p-1)/2} \,.
$$
Since this is valid for any $\Omega$, Lebesgue's differentiation theorem implies that
$$
\rho(x)^{1/2} \leq \rho_\tau(x)^{1/2} + (2\pi)^{-d/2} \left(\frac p{p-1} \right)^{1/2} \| a\|_{p,w}^{p/2} \, \tau^{-(p-1)/2} \quad \text{a.e.} \,,
$$
and therefore
$$
\rho_\tau(x) \geq \left( \rho(x)^{1/2} - (2\pi)^{-d/2} \left(\frac p{p-1} \right)^{1/2} \| a\|_{p,w}^{p/2} \, \tau^{-(p-1)/2} \right)_+^2 \quad \text{a.e.} \,.
$$

Finally, we insert this bound into \eqref{eq:repr} and compute for a.e. $x$
\begin{align*}
& \int_0^\infty \left( \rho(x)^{1/2} - (2\pi)^{-d/2} \left(\frac p{p-1} \right)^{1/2} \| a\|_{p,w}^{p/2} \, \tau^{-(p-1)/2} \right)_+^2 \, \frac{d\tau}{\tau^{2}} \\
& \quad = \rho(x)^{p/(p-1)} (2\pi)^{d/(p-1)} \left(\frac{p-1} p \right)^{p/(p-1)} \| a\|_{p,w}^{-p/(p-1)} \frac{p-1}{p+1} \,.
\qedhere
\end{align*}
\end{proof}

\begin{proof}
[Proof of Theorem \ref{main}]
If $a>0$ a.e., then Lemmas \ref{dual} and \ref{rumin} imply that
$$
\tr\left( a(-i\nabla)^{1/2} b(X) a(-i\nabla)^{1/2} - \mu\right)_+ 
\leq \mu^{-p+1} \left( \frac{p+1}{p-1} \right)^{p-1} D_{d,p}^\sc \ \| a\|_{p,w}^{p} \| b\|_p^p \,,
$$
where
$$
D_{d,p}^\sc = \frac{(p-1)^{p-1}}{p^p} \ \left( R_{d,p}^\sc \right)^{-p+1} \,.
$$
For general non-negative $a$ we apply the bound to $a_\epsilon=\max\{a,\epsilon \tilde a\}$, where $\tilde a$ is a fixed, positive function in $L_{p,\infty}(\R^d)$. Since
\begin{align*}
& \tr\left( a(-i\nabla)^{1/2} b(X) a(-i\nabla)^{1/2} - \mu\right)_+
= \tr\left( b(X)^{1/2} a(-i\nabla) b(X)^{1/2} - \mu\right)_+ \\
& \quad \leq \tr\left( b(X)^{1/2} a_\epsilon(-i\nabla) b(X)^{1/2} - \mu\right)_+
= \tr\left( a_\epsilon(-i\nabla)^{1/2} b(X) a_\epsilon(-i\nabla)^{1/2} - \mu\right)_+ \,,
\end{align*}
the assertion follows from the bound for $a_\epsilon$ and the fact that $\lim \| a_\epsilon \|_{p,w} = \|a\|_{p,w}$.
\end{proof}

%%%%%%%%%%%%%%%%%%%%%%%%%%%%%%%%%%%%%%%%%%%%%%%%%%%%%%%%

\section{Generalizations}

\subsection{An operator-valued version of Cwikel's theorem}

The works \cite{La,LaWe1} have made clear that good constants in CLR and related inequalities in higher dimensions can be derived from operator-valued versions of these inequalities in lower dimensions. In the case of Cwikel's theorem this strategy was implemented in \cite{H1}. The constant in the CLR inequality for Schr\"odinger operators with matrix-valued potentials was improved in \cite{FrLiSe1}. In this subsection we show that Rumin's proof can also be modified to yield an operator-valued version of Cwikel's theorem. This extension is not straightforward and leads, unfortunately, to a somewhat worse constant than that in Corollary \ref{cwikel}.

Another thing that we show in this subsection is that the structure of $\R^d$ is not really relevant for Cwikel's theorem. Indeed, our theorem holds on a general pair of measure spaces, with the role of the Fourier transform being played by a general unitary operator with bounded integral kernel. Results in this spirit have already appeared in \cite{BiKaSo}, but it is not clear whether their techniques also apply in the operator-valued case.

We begin with some notations. In this subsection, let $(X,dx)$ and $(Y,dy)$ be sigma-finite measure spaces and let $\mathcal H$ and $\mathcal G$ be separable Hilbert spaces. We denote by $L_{p}(X,\S_p(\mathcal H))$ the space of all measurable functions $f$ on $X$ with values in the compact operators in $\mathcal H$ such that
$$
\| f\|_{L_p(\S_p)}^p = \int_X \| f(x)\|_{\S_p(\mathcal H)}^p \,dx <\infty \,.
$$
Similarly, $L_{p,w}(Y,\B(\mathcal G))$ is the space of all measurable functions $g$ on $Y$ with values in the bounded operators on $\mathcal G$ such that
$$
\| g\|_{L_{p,w}(\B)}^p = \sup_{\tau>0} \tau^p \, |\{ y\in Y:\ \|g(y)\|_{\B(\mathcal G)} >\tau \} |  \,. 
$$

\begin{theorem}
[Operator-valued version of Cwikel's theorem]
\label{cwikelop}
Let $\Phi: L_2(X,\mathcal H)\to L_2(Y,\mathcal G)$ be a unitary operator, which maps $L_1(X,\mathcal H)$ boundedly into $L_\infty(Y,\mathcal G)$. Let $q>2$. If $f\in L_{q}(X,\S_q(\mathcal H))$, $g\in L_{q,w}(Y,\B(\mathcal H))$, then $f\,\Phi^*\,g\in\S_{q,w}( L_2(Y,\mathcal H),L_2(X,\mathcal G) )$ with
$$
\| f\, \Phi^* \, g \|_{q,w}^q \leq \frac{q}2 \left( \frac{q}{q-2} \right)^{q-1} C^2 \ \| f \|_{L_{q}(\S_q)}^{q} \ \| g\|_{L_{q,w}(\B)}^q \,,
$$
where $C = \|\Phi\|_{L_1\to L_\infty}$.
\end{theorem}

This constant is worse than that of Corollary \ref{cwikel} by a factor of $\frac{q}2 \left( \frac{q}{q+2} \right)^{(q-2)/2}>1$. It is still better, by a factor of $2^{2q-5} q$, than the constant
$$
\left( \frac{q}{2} \right)^q \left( \frac 8{q-2} \right)^{q-2} \frac q{q-2}\ C^2 \,.
$$
from \cite{H1} (which is the same as in \cite{Cw} in the scalar case).

\begin{proof}
The heart of the proof is the following analogue of Lemma \ref{rumin}. Namely, if $p>1$ and $a\in L_{p,w}(Y,\B(\mathcal G))$ with $a(y)\geq 0$ and $\ker a(y)=\{0\}$ for a.e. $y\in Y$, then for any operator $\gamma$ on $L_2(X,\mathcal H)$ satisfying $0 \leq \gamma \leq \Phi^*a\Phi$,
\begin{equation}
\label{eq:cwikelopgoal2}
\tr\gamma^{1/2} \Phi^* a^{-1} \Phi \gamma^{1/2} \geq \frac{(p-1)^{(2p-1)/(p-1)}}{p^{2p/(p-1)}}\ C^{-2/(p-1)} \  \| a\|_{L_{p,w}(\B)}^{-p'} \int_X \tr_{\mathcal H} \gamma(x,x)^{p'} \,dx \,.
\end{equation}
Here again $C=\|\Phi\|_{L_1\to L_\infty}$.

Accepting \eqref{eq:cwikelopgoal2} for the moment, we briefly explain how to finish the proof of Theorem \ref{cwikelop}. First, via a straightforward extension of Lemma \ref{dual} we infer from \eqref{eq:cwikelopgoal2} that
\begin{equation}
\label{eq:cwikelopgoal}
\tr\left( a^{1/2} \Phi b \Phi^* a^{1/2} - \mu\right)_+  \leq \mu^{-p+1} \left( \frac{p}{p-1}\right)^p C^2 \ \| a\|_{L_{p,w}(\B)}^{p} \| b\|_{L_p(\S_p)}^p \,,
\end{equation}
provided that $a(y)$ and $b(x)$ are non-negative for a.e. $x$ and $y$. This implies Theorem~\ref{cwikelop} in the same way as Theorem \ref{main} implied Corollary \ref{cwikel}.

We now turn to the proof of \eqref{eq:cwikelopgoal2}. We write, similarly as before,
\begin{equation}
\label{eq:reprop}
\tr\gamma^{1/2} \Phi^* a^{-1}\Phi\gamma^{1/2} = \int_{X} \int_0^\infty \tr_\mathcal H \gamma_\tau(x,x) \ \frac{d\tau}{\tau^{2}} \,dx
\end{equation}
with $\gamma_\tau = P_\tau \gamma P_\tau$ and $P_\tau = \Phi^*\chi_{(0,\tau]}(a)\Phi$.

For any Hilbert--Schmidt operator $H$ in $\mathcal H$, any set $\Omega\subset X$ of finite measure and any $\epsilon>0$, we apply the Schwarz inequality to find
\begin{align*}
\int_\Omega \tr_{\mathcal H} H^* \gamma(x,x) H \,dx & = \tr_{L_2(X,\mathcal H)} \chi_\Omega H^* \gamma H \chi_\Omega \\
& \leq (1+\epsilon) \tr_{L_2(X,\mathcal H)} \chi_\Omega H^* P_\tau \gamma P_\tau H \chi_\Omega \\
& \qquad + (1+\epsilon^{-1}) \tr_{L_2(X,\mathcal H)} \chi_\Omega H^* P_\tau^\bot \gamma P_\tau^\bot H \chi_\Omega \,.
\end{align*}
(Here $H \chi_\Omega$ is short for $\chi_\Omega\otimes H$ and $P_\tau^\bot$ for $1-P_\tau$.) For the first term on the right side, we notice that
$$
\tr_{L_2(X,\mathcal H)} \chi_\Omega H^* P_\tau \gamma P_\tau H \chi_\Omega = \int_\Omega \tr_{\mathcal H} H^* \gamma_\tau(x,x) H \,dx \,.
$$
In order to bound the second term we recall the fact that $\gamma\leq \Phi^* a\Phi$ and that $C=\|\Phi\|_{L_1\to L_\infty}<\infty$, which yields
\begin{align*}
\tr_{L_2(X,\mathcal H)} \chi_\Omega H^* P_\tau^\bot \gamma P_\tau^\bot H \chi_\Omega
& \leq \tr_{L_2(X,\mathcal H)} \chi_\Omega H^*  \Phi^* a \chi_{\{a>\tau\}} \Phi H \chi_\Omega \\
& = \int_\Omega \int_Y \tr_\mathcal H H^* \Phi(y,x)^* a(y) \chi_{\{a(y)>\tau\}} \Phi(y,x) H \,dy\,dx \\
& \leq \int_\Omega \int_Y \| a(y) \chi_{\{a(y)>\tau\}} \|_{\B} \| \Phi(y,x) \|_{\B}^2 \tr_\mathcal H H^* H \,dy\,dx \\
& \leq |\Omega|\, C^2 \tr_{\mathcal H} H^* H \ \int_Y \| a(y)\|_{\B} \ \chi_{\{\|a(y)\|_{\B}>\tau\}} \,dy \,.
\end{align*}
Here we used the fact that $\| a(y) \chi_{\{a(y)>\tau\}} \|_{\B}= \| a(y)\|_{\B} \, \chi_{\{\|a(y)\|_{\B}>\tau\}}$. Now the same weak $L_p$ bound as in the proof of Lemma \ref{rumin} leads to
$$
\tr_{L_2(X,\mathcal H)} \chi_\Omega H^* P_\tau^\bot \gamma P_\tau^\bot H \chi_\Omega
\leq |\Omega|\, C^2 \tr_{\mathcal H} H^* H \ \frac p{p-1} \| a\|_{L_{p,w}(\B)}^p \tau^{-p+1} \,.
$$
To summarize, we have shown that
\begin{align*}
\int_\Omega \tr_{\mathcal H} H^* \gamma(x,x) H \,dx \leq & (1+\epsilon) \int_\Omega \tr_{\mathcal H} H^* \gamma_\tau(x,x) H \,dx \\
& + (1+\epsilon^{-1}) |\Omega|\, C^2 \tr_{\mathcal H} H^* H \ \frac p{p-1} \| a\|_{L_{p,w}(\B)}^p \tau^{-p+1} \,.
\end{align*}
Since this is valid for any $\Omega$ and for any $H$, we have for a.e. $x\in X$ the operator inequality
$$
\gamma(x,x) \leq (1+\epsilon) \gamma_\tau(x,x) + (1+\epsilon^{-1}) C^2 \frac p{p-1} \| a\|_{L_{p,w}(\B)}^p \tau^{-p+1} \,.
$$
We now use the fact that an operator inequality $A\geq B$ implies $\tr f(A)\geq \tr f(B)$ for $f$ non-decreasing. In our case $f(t)=t_+$, the positive part, and therefore
$$
\tr_{\mathcal H} \gamma_\tau(x,x) \geq (1+\epsilon)^{-1} \tr_{\mathcal H} \left( \gamma(x,x) - (1+\epsilon^{-1}) C^2 \frac p{p-1} \| a\|_{L_{p,w}(\B)}^p \tau^{-p+1} \right)_+ \,.
$$
It remains to do the $\tau$ integration,
$$
\int_0^\infty \tr_{\mathcal H} \gamma_\tau(x,x) \,\frac{d\tau}{\tau^2} \geq 
\frac{\epsilon^{1/(p-1)}}{(1+\epsilon)^{p/(p-1)}}
\left( \frac{p-1}{p} \right)^{p/(p-1)} C^{-2/(p-1)} \| a\|_{L_{p,w}(\B)}^{-p'} \tr_{\mathcal H} \gamma(x,x)^{p'} \,,
$$
and to optimize in $\epsilon$ by choosing $\epsilon=(p-1)^{-1}$. This, together with \eqref{eq:reprop} proves \eqref{eq:cwikelopgoal2} and completes the proof.
\end{proof}

%%%%%%%%%%%%%%%%%%%%%%%%%%%%%%%%%%%%%%%%%%%%%%%%%%%%

\subsection{The CLR inequality for general Schr\"odinger-like operators}

Next, we show that for a large class of `kinetic energies' $T$ the number $N(0,T+V)$ of negative eigenvalues (counting multiplicities) of the Schr\"odinger-type operator $T+V$ can be bounded in terms of an integral of the potential $V$. We shall see how the exponent with which $V$ enters into this bound is determined by $T$. The improvement of this result as compared to those in \cite{LeSo,FrLiSe2} is that we do not require the potential to be scalar and that we do not require $\exp(-tT)$ to be positivity preserving.

Again, throughout this subsection we assume that $X$ is a sigma-finite measure space and $\mathcal{H}$ a separable Hilbert spaces.

\begin{theorem}
\label{cwikelgen}
Let $T$ be a non-negative operator in $L_2(X,\mathcal H)$ with $\ker T=\{0\}$. Assume that there are constants $\nu>2$ and $A<\infty$ such that for every $E>0$, every $\Omega\subset X$ of finite measure and every $\phi\in\mathcal H$,
\begin{equation}
\label{eq:ass}
\tr_{L_2(X)} \chi_\Omega \left(\phi, T^{-1} \chi_{(0,E]}(T) \phi \right)_{\mathcal H} \chi_\Omega \leq A E^{(\nu-2)/2} |\Omega| \|\phi\|_{\mathcal H}^2 \,. 
\end{equation}
Then for any measurable function $V$ on $X$, taking values in the self-adjoint compact operators on $\mathcal H$,
$$
N(0,T+V) \leq C_\nu \, A \int_X \tr_{\mathcal{H}} V(x)_-^{\nu/2} \,dx
$$
with
$$
C_\nu = \frac\nu 2 \left( \frac{\nu}{\nu-2} \right)^{\nu-2} \,.
$$
If $\dim\mathcal H=1$, then $C_\nu$ can be replaced by
$$
C_\nu = \left( \frac{\nu(\nu+2)}{(\nu-2)^2} \right)^{(\nu-2)/2} \,.
$$
\end{theorem}

Roughly speaking, assumption \eqref{eq:ass} means that $T^{-1} \chi_{(0,E]}(T)$ has an integral kernel (taking values in the bounded operators on $\mathcal H$) which on the diagonal satisfies the bound
$$
\| T^{-1} \chi_{(0,E]}(T)(x,x) \|_{\mathfrak B(\mathcal H)} \leq A E^{(\nu-2)/2}
$$
We discuss the equivalence of this assumption with more standard assumptions in Lemma \ref{ass} below.

Before turning to the proof of Theorem \ref{cwikelgen} we illustrate it by

\begin{example}
Let $T=(-\Delta)^s$, $0<s<d/2$, in $L_2(\R^d)$. Then by explicit diagonalization via Fourier transform one sees that \eqref{eq:ass} holds with $\nu=d/s$ and
$$
A= \int_{\R^d} |p|^{-2s} \chi_{\{|p|<1\}} \frac{dp}{(2\pi)^d} 
= \frac{\omega_d}{(2\pi)^{d}} \frac{d}{d-2s} \,.
$$
Thus Theorem \ref{cwikelgen} implies that
\begin{equation}
\label{eq:clrrdop}
N(0,(-\Delta)^s +V) \leq \frac{d}{2s} \left( \frac{d}{d-2s} \right)^{(d-2s)/s} \frac{\omega_d}{(2\pi)^{d}} \frac{d}{d-2s} \int_{\R^d} \tr_\mathcal H V(x)_-^{d/2s}\,dx
\end{equation}
in the operator-valued case and
\begin{equation}
\label{eq:clrrdscal}
N(0,(-\Delta)^s +V) \leq \left( \frac{d (d+2s)}{(d-2s)^2} \right)^{(d-2s)/2s} \frac{\omega_d}{(2\pi)^{d}} \frac{d}{d-2s} \int_{\R^d} V(x)_-^{d/2s}\,dx
\end{equation}
in the scalar case. These constants are rather good. In the cases which are most relevant in applications the bounds are about a factor of two worse than the best available bounds. Indeed, for $d=3$, \eqref{eq:clrrdscal} gives 0.196 for $s=1$ (to be compared with 0.116 from \cite{Li}) and 0.228 for $s=1/2$ (to be compared with 0.103 from \cite{Da}) and \eqref{eq:clrrdop} gives 0.228 for $s=1$ (to be compared with 0.174 from \cite{FrLiSe1}). We emphasize again that the methods of \cite{Li,Da,FrLiSe1} are restricted to $s\leq 1$. The above constants are the best ones available for $1<s<d/2$; see the comparison with the constant from \cite{Cw,H1} after Theorem \ref{cwikelop}.
\end{example}

\begin{proof}
By the variational principle and the Birman--Schwinger principle,
$$
N(0,T+V) \leq N(0,T-V_-) = n(1,T^{-1/2} V_- T^{-1/2}) \,.
$$
Thus, by the same argument as in the proof of Lemma \ref{equiv}, Theorem \ref{cwikelgen} will follow if we can show that
$$
\tr \left( T^{-1/2} V_- T^{-1/2} -\mu \right)_+ \leq \mu^{-\nu/2 +1} A D \int_X \tr_\mathcal H V(x)_-^{\nu/2} \,dx \,.
$$
Here, $D= (\nu/(\nu-2))^{(\nu-2)/2}$ in the general case, which can be improved to $D=(2/\nu)((\nu+2)/(\nu-2))^{(\nu-2)/2}$ for $\dim\mathcal H=1$. By the argument of Lemma \ref{dual} the latter inequality is, in turn, equivalent to the inequality
$$
\tr \gamma^{1/2} T \gamma^{1/2} \geq A^{-2/(\nu-2)} K \int_X \tr_{\mathcal H} \gamma(x,x)^{\nu/(\nu-2)} \,dx
$$
for every operator $0\leq\gamma\leq T^{-1}$. Here
$$
K= \frac{2^{2/(\nu-2)}(\nu-2)^2}{\nu^{2(\nu-1)/(\nu-2)}}
$$
in the general case, which can be improved to $K=(\nu-2)^2/(\nu(\nu+2))$ for $\dim\mathcal H=1$. In the scalar case $\dim\mathcal H=1$, this bound follows from \cite{Ru1} (with the improved constant of \cite{Ru}) and the modifications to treat the general case are similar to our arguments in the proof of Theorem \ref{cwikelop}.

For the sake of completeness, we briefly sketch the proof. We introduce $P_E=\chi_{(E,\infty)}(T)$ and $P_E^\bot=\chi_{(0,E]}(T)$. The key is, as before, the bound
$$
\tr_{ L_2(X,\mathcal H)} \chi_\Omega H^* P_E^\bot \gamma P_E^\bot H \chi_\Omega
\leq \tr_{ L_2(X,\mathcal H)} \chi_\Omega H^* P_E^\bot T^{-1} P_E^\bot H \chi_\Omega
$$
for any set $\Omega\subset X$ of finite measure and any Hilbert--Schmidt operator $H$ on $\mathcal H$. By assumption \eqref{eq:ass} the right side is bounded by $A E^{(\nu-2)/2} |\Omega| \tr_{\mathcal H} H^*H$. This implies, as before,
$$
\gamma(x,x) \leq (1+\epsilon) \left( P_E\gamma P_E \right) (x,x) + \left(1+\epsilon^{-1}\right) A E^{(\nu-2)/2}
$$
for every $\epsilon>0$. In the special case $\dim\mathcal H=1$ the bound can be somewhat improved using the argument of Lemma \ref{rumin} to
$$
\sqrt{\gamma(x,x)} \leq \sqrt{\left( P_E\gamma P_E \right) (x,x)} + A^{1/2} E^{(\nu-2)/4} \,.
$$
With these bounds at hand the proof is completed as before by integration over $E$.
\end{proof}

We now give sufficient conditions for assumption \eqref{eq:ass}, which can be verified in applications. Similar results are contained in \cite{Ru1}.

\begin{lemma}
\label{ass}
Let $T$ be a non-negative operator in $L_2(X,\mathcal H)$, let $\Omega\subset X$ have finite measure and let $\phi\in\mathcal H$.
If, for some constants $\nu>0$ and $C'$ and all $t>0$
\begin{equation}
\label{eq:assheat}
\tr_{L_2(X)} \chi_\Omega \left(\phi, \exp(-tT) \phi \right)_{\mathcal H} \chi_\Omega \leq C' t^{-\nu/2} \,,
\end{equation}
then for all $E>0$
\begin{equation}
\label{eq:assproj}
\tr_{L_2(X)} \chi_\Omega \left(\phi, \chi_{(0,E]}(T) \phi \right)_{\mathcal H} \chi_\Omega \leq B' E^{\nu/2}
\end{equation}
with $B'=C'( 2e/\nu)^{\nu/2}$. Moreover, if \eqref{eq:assproj} holds for some constants $\nu>2$ and $B'$ and all $E>0$, then for all $E>0$
\begin{equation}
\label{eq:asslem}
\tr_{L_2(X)} \chi_\Omega \left(\phi, T^{-1} \chi_{(0,E]}(T) \phi \right)_{\mathcal H} \chi_\Omega \leq A' E^{(\nu-2)/2}
\end{equation}
with $A'=B'\nu/(\nu-2)$.
\end{lemma}

\begin{proof}
To prove the first assertion of the lemma we use the bound $\chi_{(0,E]}(\lambda) \leq e^{t E} e^{-t\lambda}$. Thus \eqref{eq:assheat} implies
$$
\tr_{L_2(X)} \chi_\Omega \left(\phi, \chi_{(0,E]}(T) \phi \right)_{\mathcal H} \chi_\Omega \leq C' t^{-\nu/2} e^{t E}
$$
for all $t>0$. We optimize the right side by choosing $t=\nu/(2E)$.

To prove the second assertion we write
$$
\lambda^{-1} \chi_{(0,E]}(\lambda) = \int_0^\infty \chi_{(0,\min\{s,E\}]}(\lambda) \frac{ds}{s^2} \,.
$$
Thus \eqref{eq:assproj} implies
$$
\tr_{L_2(X)} \chi_\Omega \left(\phi, T^{-1} \chi_{(0,E]}(T) \phi \right)_{\mathcal H} \chi_\Omega
\leq B' \int_0^\infty \min\{s,E\}^{\nu/2} \frac{ds}{s^2}
= B' \frac{\nu}{\nu-2} E^{(\nu-2)/2} \,,
$$
as claimed.
\end{proof}

Assumption \eqref{eq:assheat} is a standard assumption in works on ultra-contractivity. In the work of Levin and Solomyak \cite{LeSo} (see also \cite{FrLiSe2}) it was used to extend the proof of Li and Yau \cite{LiYa} to general Dirichlet forms generating submarkovian semi-groups. The important difference, however, is that here, as in \cite{Ru1,Ru}, we do \emph{not} need the heat kernel to be positivity preserving and a contraction on $L_1$.

One application of Lemma \ref{ass} concerns magnetic Schr\"odinger operators. That is, take $X=\R^d$, $\mathcal H=\C$ and $T=(-i\nabla +A)^2$ for some $A\in L_{2,\loc}(\R^d,\R^d)$. While we do not know how to verify \eqref{eq:ass} directly, we know from the diamagnetic inequality that \eqref{eq:assheat} holds with $C'= (4\pi)^{-d/2} |\Omega| \|\phi\|_{\mathcal H}^2$. Thus, in dimension $d\geq 3$, \eqref{eq:asslem} holds with $A'= (e/(2\pi d))^{d/2} (d/(d-2)) |\Omega| \|\phi\|_{\mathcal H}^2$ and \eqref{eq:ass} with $A= (e/(2\pi d))^{d/2} (d/(d-2))$. While this constant is worse than that without magnetic field, it is independent of the magnetic field, as it should be.

%%%%%%%%%%%%%%%%%%%%%%%%%%%%%%%%%%%%%%%%%%%%%

\section{Concluding remarks}\label{sec:concl}

In this final subsection we discuss the problem of finding the optimal (i.e., largest possible) constant $K_{s,d}$ in Rumin's inequality
\begin{equation}
\label{eq:const}
\tr\gamma^{1/2}(-\Delta)^s\gamma^{1/2} \geq K_{s,d} \int_{\R^d} \gamma(x,x)^{d/(d-2s)} \,dx
\end{equation}
for operators $\gamma$ on $L_2(\R^d)$ satisfying $0\leq\gamma\leq(-\Delta)^{-s}$. We assume throughout that $2s<d$.

Lemma \ref{rumin} (with $a(\xi)=|\xi|^{-2s}$ and $p=d/2s$) implies that this inequality holds and that the optimal constant satisfies
$$
K_{s,d} \geq \frac{d-2s}{d+2s} \ (2\pi)^{2ds/(d-2s)} \left( \frac{d-2s}{d} \right)^{d/(d-2s)} \omega_d^{-2s/(d-2s)}
$$
Here $\omega_d=|\{\xi\in\R^d:\ |\xi|<1\}|$. In the following subsections we derive two upper bounds for $K_{s,d}$ and discuss a non-obvious symmetry.

%%%%%%%%%%%%%%%%%%%%%%%%%%%%%%%%%%%%%%%%%%%%%%%

\subsection{The semi-classical constant}

Here we show that
\begin{equation}
\label{eq:constbdsc}
K_{s,d} \leq (2\pi)^{2ds/(d-2s)} \left( \frac{d-2s}{d} \right)^{d/(d-2s)} \omega_d^{-2s/(d-2s)} \,.
\end{equation}
Note that this upper bound differs from the constant in Lemma \ref{rumin} only by a factor of $(d-2s)/(d+2s)$. There are two ways to prove \eqref{eq:constbdsc}. The first one consists in noting that a Weyl-type semi-classical formula yields a lower bound on the optimal constant $D_{s,d}$ in the inequality
$$
\tr\left((-\Delta)^{-s/2} V_- (-\Delta)^{-s/2} - \mu \right)_+ \leq D_{s,d} \ \mu^{-d/2s+1} \int_{\R^d} V(x)_-^{d/2s} \,dx
$$
and then using Lemma \ref{dual} to convert this into an upper bound on $K_{s,d}$. Since this is standard, we explain a less known, but more direct approach. Instead of finding the best constant $K_{s,d}$ in \eqref{eq:const} we look for the best constant $K_{s,d}'$ in the inequality
\begin{equation}
\label{eq:constsc}
\iint_{\R^d\times\R^d}  |p|^{2s} M(p,x) \frac{dp\,dx}{(2\pi)^d}
\geq K_{s,d}' \int_{\R^d} \left( \int_{\R^d} M(p,x) \frac{dp}{(2\pi)^{d}} \right)^{d/(d-2s)} \,dx
\end{equation}
for all functions $M$ on $\R^d\times\R^d$ satisfying $0\leq M(p,x) \leq |p|^{-2s}$ for all $x$ and $p$. Using coherent states it is easy to verify that $K_{s,d}\leq K_{s,d}'$. It is elementary to compute the optimal constant $K_{s,d}'$. It is given by the right side of \eqref{eq:constbdsc}. Optimizers $M$ are of the form $M(p,x) = |p|^{-2s}\chi_{\{|p|<R(x)\}}$ for an arbitrary function $R$.

%%%%%%%%%%%%%%%%%%%%%%%%%%%%%%%%%%%%%%%%%%%%%

\subsection{The Sobolev constant}

Applying \eqref{eq:const} to an operator $\gamma=\alpha |\psi\rangle\langle\psi|$ of rank one with $\alpha = \|(-\Delta)^{s/2}\psi\|^{-2}$ we obtain
\begin{equation}
\label{eq:sob}
\left\|(-\Delta)^{s/2}\psi\right\|^2 \geq K_{s,d}^{(d-2s)/d} \left( \int_{\R^d} |\psi|^{2d/(d-2s)} \,dx \right)^{(d-2s)/d} \,.
\end{equation}
This is Sobolev's inequality. The best constant in this inequality for general $s$ has been determined by Lieb \cite{Li2} (in a dual formulation). Using this  value, we infer that
\begin{equation}
\label{eq:constbdsob}
K_{s,d} \leq (4\pi)^{ds/(d-2s)} \left( \frac{\Gamma((d+2s)/2)}{\Gamma((d-2s)/2)} \right)^{d/(d-2s)} \left( \frac{\Gamma(d/2)}{\Gamma(d)} \right)^{2s/(d-2s)} \,.
\end{equation}
Numerically, it is easy to determine which one of the upper bounds \eqref{eq:constbdsc} and \eqref{eq:constbdsob} is better. It seems like \eqref{eq:constbdsob} is better for $d=1$ and \eqref{eq:constbdsc} is better for $d\geq 3$. In $d=2$, \eqref{eq:constbdsob} is better for $s<1/2$ and \eqref{eq:constbdsc} is better for $s>1/2$. We also remark that the constants on the right sides of \eqref{eq:constbdsc} and \eqref{eq:constbdsob} are asymptotically equal as $s\to 0$ and as $s\to d/2$.

%%%%%%%%%%%%%%%%%%%%%%%%%%%%%%%%%%%%%%%%%%%%%

\subsection{Conformal invariance}

Lieb \cite{Li2} has shown that \eqref{eq:sob} (or an equivalent version thereof) is conformally invariant in the following sense. If $h$ is a conformal transformation of $\R^d\cup\{\infty\}$ and if $\phi(x) = J_h(x)^{(d-2s)/2d} \psi(h(x))$, where $J_h$ is the Jacobian of $h$, then
$$
\left\|(-\Delta)^{s/2}\phi\right\|^2 = \left\|(-\Delta)^{s/2} \psi\right\|^2
\qquad\text{and}\qquad
\int_{\R^d} |\phi|^{2d/(d-2s)} \,dx = \int_{\R^d} |\psi|^{2d/(d-2s)} \,dx \,.
$$

Similarly, we now argue that \eqref{eq:const} is conformally invariant
under replacing $\gamma(x,y)$ by $J_h(x)^{(d-2s)/2d} \gamma(h(x),h(y)) J_h(y)^{(d-2s)/2d}$. We first observe that \eqref{eq:const} is equivalent to the following inequality. For any sequence of functions $(\psi_j)\subset \dot H^s(\R^d)$ satisfying $\langle(-\Delta)^{s/2}\psi_j,(-\Delta)^{s/2}\psi_k\rangle = \delta_{j,k}$ and for any sequence of numbers $(\lambda_j)$ satisfying $0\leq\lambda_j\leq 1$, we have
$$
\sum_j \lambda_j \geq K_{s,d} \int_{\R^d} \left( \sum_j \lambda_j |\psi_j|^2 \right)^{d/(d-2s)} \,dx \,.
$$
This equivalence follows by expanding the trace class operator $(-\Delta)^{s/2}\gamma(-\Delta)^{s/2}=\sum_j \lambda_j |f_j\rangle\langle f_j|$ into its eigenfunctions and setting $\psi_j=(-\Delta)^{-s/2}f_j$.

If we now let $\phi_j(x) = J_h(x)^{(d-2s)/2d} \psi_j(h(x))$, then, by polarization of the above identity,
$$
\langle(-\Delta)^{s/2}\phi_j,(-\Delta)^{s/2}\phi_k\rangle
= \langle(-\Delta)^{s/2}\psi_j,(-\Delta)^{s/2}\psi_k\rangle \,,
$$
and clearly
$$
\int_{\R^d} \left( \sum_j \lambda_j |\phi_j|^2 \right)^{d/(d-2s)} \,dx = \int_{\R^d} \left( \sum_j \lambda_j |\psi_j|^2 \right)^{d/(d-2s)} \,dx \,.
$$
This proves that Rumin's inequality \eqref{eq:const} is invariant under replacing $\gamma(x,y)$ by $J_h(x)^{(d-2s)/2d} \gamma(h(x),h(y)) J_h(y)^{(d-2s)/2d}$ for any conformal transformation.

One consequence of this conformal invariance is that the inequality has an equivalent formulation on the sphere $\Sph^d$ via stereographic projection as in \cite{Li2}. In light of previous results about conformally invariant trace inequalities \cite{Mo} it is natural to wonder about the sharp constant in \eqref{eq:const}.

%%%%%%%%%%%%%%%%%%%%%%%%%%%%%%%%%%%%%%%%%%%%%%%%%%%%%%%%%%%%%%%%%%%%%%%%%%%%%%%%%%%%%%%%%%%

\bibliographystyle{amsalpha}

\end{document}